\documentclass[a4]{amsart}

\input xypic
\usepackage{amssymb}
\usepackage{MnSymbol}
\usepackage{bbm}
\usepackage{amsmath}

\newtheorem{theorem}{Theorem}[section]

\newtheorem{proposition}[theorem]{Proposition}
\newtheorem{lemma}[theorem]{Lemma}
\newtheorem{corollary}[theorem]{Corollary}

\theoremstyle{definition}
\newtheorem{definition}[theorem]{Definition}
\newtheorem{remark}[theorem]{Remark}

\newtheorem{example}[theorem]{Example}

\newcounter{bean}
\newenvironment{letterlist}{\begin{list}{\rm ({\alph{bean}})}
      {\usecounter{bean}\setlength{\rightmargin}{\leftmargin}}}
      {\end{list}}
\newenvironment{romanlist}{\begin{list}{\rm ({\roman{bean}})}
      {\usecounter{bean}\setlength{\rightmargin}{\leftmargin}}}
      {\end{list}}

\newcommand{\seqm}[3]{\ensuremath{#1\stackrel{#2}
 {\longrightarrow}#3}}
\newcommand{\seqmm}[5]{\ensuremath{#1\stackrel{#2}
 {\longrightarrow}#3\stackrel{#4}{\longrightarrow}#5}}
\newcommand{\seqmmm}[7]{\ensuremath{#1\stackrel{#2}
 {\longrightarrow}#3\stackrel{#4}{\longrightarrow}#5
  \stackrel{#6}{\longrightarrow}#7}}

\newcommand{\larrow}{\relbar\!\!\relbar\!\!\rightarrow}
\newcommand{\llarrow}{\relbar\!\!\relbar\!\!\larrow}

\newcommand{\llnamedright}[3]{\ensuremath{#1\stackrel{#2}
 {\llarrow}#3}}

\newcommand{\llnamedddright}[7]{\ensuremath{#1\stackrel{#2}
 {\llarrow}#3\stackrel{#4}{\llarrow}#5
  \stackrel{#6}{\llarrow}#7}}

\newcommand{\paren}[1]{\ensuremath{\left\{ #1 \right\}}}

\newcommand{\sbracket}[1]{\ensuremath{\left[#1\right]}}

\newcommand{\cplus}[3]{\displaystyle\bigoplus^{#2}_{#1}#3}
\newcommand{\csum}[3]{\displaystyle\sum^{#2}_{#1}#3}

\newcommand{\qqed}{\hfill\square}

\newcommand{\mc}[1]{\ensuremath{\mathcal{#1}}}
\newcommand{\mb}[1]{\ensuremath{\mathbb{#1}}}

\newcommand{\Z}{\ensuremath{\mathbb{Z}}}
\newcommand{\ID}{\ensuremath{\mathbbm{1}}}

\def\dim{\mathrm{dim}}

\begin{document}

\title[Loop space homotopy types] 
         {The loop space homotopy type of simply-connected four-manifolds 
         and their generalizations} 
\author{Piotr Beben}
\address{School of Mathematics, University of Southampton,  
        Southampton SO17 1BJ, United Kingdom} 
\email{P.D.Beben@soton.ac.uk} 
\author{Stephen Theriault}
\address{School of Mathematics, University
         of Southampton, Southampton SO17 1BJ, United Kingdom}
\email{S.D.Theriault@soton.ac.uk}


\begin{abstract} 
We determine loop space decompositions of simply-connected four-manifolds, 
$(n-1)$-connected $2n$-dimensional manifolds provided $n\notin\{4,8\}$, 
and connected sums of products of two spheres. These are obtained as 
special cases of a more general loop space decomposition of certain 
torsion-free $CW$-complexes with well-behaved skeleta and some  
Poincar\'{e} duality features.   
\end{abstract}

\maketitle 

\section{Introduction} 

The topology of simply-connected four-manifolds is a subject 
of widespread and enduring interest. They have been classified up to 
homotopy type by Milnor~\cite{Mi} and up to homeomorphism type 
by Freedman~\cite{F}. Their classification up to diffeomorphism 
type is one of the great unsolved questions in modern mathematics, 
with significant advances achieved by Donaldson~\cite{D} and 
Seiberg and Witten~\cite{Wi}. They have also been studied in view 
of their connections to other areas of mathematics, such as 
knot theory~\cite{FS} and symplectic geometry~\cite{P}. 

The homotopy theory of simply-connected four-manifolds has 
continued to attract considerable attention since Milnor's classification. 
For example, given simply-connected four-manifolds $M$ and~$N$, 
Cochran and Habegger~\cite{CH} calculated the group of 
self-homotopy equivalences of $M$; Zhao, Gao and Su~\cite{ZGS} 
calculated the homotopy classes of maps $[M,N]$; and Baues~\cite{B} 
has written a monograph entirely devoted to investigating the 
homotopy theory of $M$, $N$ and the maps between them. 

In another direction, Wall~\cite{Wa} initiated the study of $(n-1)$-connected 
$2n$-dimensional manifolds as a generalization of simply-connected 
four-manifolds. Such manifolds have received considerable recent attention 
as certain families of them arise as intersections of quadrics in geometric 
topology and moment-angle manifolds in toric topology~\cite{BM,GL}. 
Another variation is connected sums of products of two spheres, which 
generalizes the sub-collection of simply-connected four-manifolds that are 
connected sums of $S^{2}\times S^{2}$. Such connected sums appear 
in the classification by McGavran~\cite{Mc} of $n$-torus actions on closed, 
compact, simply-connected $(n+2)$-manifolds, and they also appear as 
intersections of quadrics and moment-angle manifolds~\cite{BM,GL}. 

In this paper we study simply-connected four-manifolds and their 
generalizations from a new perspective. Let $M$ be a simply-connected manifold. 
Let $\Omega M$ be the space of continuous basepoint preserving maps from 
the circle to~$M$, called the \emph{(based) loop space} of $M$. When $M$ 
is a simply-connected four-manifold, an $(n-1)$-connected $2n$-manifold, 
or a connected sum of  products of two spheres, we aim to give an explicit, 
integral homotopy decomposition of $\Omega M$ as a product of simpler factors. 

Decomposing the loops on large classes of manifolds has long been thought 
to be too hard to do. However, the methods used in the paper are relatively 
accessible and flexible. Essentially, the starting input is information about 
the integral cohomology of $M$ derived from Poincar\'{e} duality. This is then 
manipulated by creating appropriate homotopy fibrations involving $M$ which 
allow one to apply decomposition methods from homotopy theory, in the 
spirt of~\cite{CMN}. It should be the case that the same methods can be 
used to investigate the loops on other classes of manifolds. 

Such decompositions are useful and to illustrate this, we give three examples. 
First, in toric topology one associates to a simplicial complex $K$ a space 
called a moment-angle complex. If $K$ is a simple polytope then 
this moment-angle complex is a manifold. For example, if $K$ is 
an $n$-gon then the moment-angle complex is a connected sum of 
products of two spheres~\cite{BM,GL}. The combinatorics of the 
polytope and the geometry of the manifold are deeply connected, 
but the relationship is not well understood. Decomposing the loops 
on such connected sums and relating the factors to the 
combinatorics of the polytope should be insightful. 
Second, string topology is concerned with 
properties of the free loop space $\Lambda M$ of $M$: the space of 
continuous unbased maps from the circle to $M$. There is a fibration 
\(\seqmm{\Omega M}{}{\Lambda M}{e}{M}\) 
where $e$ evaluates a map at the basepoint, and $e$ has a section. 
The section implies that 
$\pi_{m}(\Lambda M)\cong\pi_{m}(M)\oplus\pi_{m}(\Omega M)$ 
for $m\geq 2$, so the homotopy groups of~$\Lambda M$ can be 
determined to the same extent as those of the factors of $\Omega M$. 
This has implications for counting the geodesics on $M$ (see~\cite{SV}),  
and the decomposition of $\Omega M$ may help clarify the 
homology and cohomology of $\Lambda M$. The third application  
is to configuration spaces, which will be discussed in more detail 
in Section~\ref{sec:config}. Let $F(M,k)$ be the configuration space of 
ordered $k$-tuples of distint points in the product space $M^{k}$. 
In certain cases, for example if $M$ is a product of two non-trivial 
manifolds, Cohen and Gitler~\cite{CG} showed that $\Omega M$ is a 
factor of $\Omega F(M,k)$. A decomposition for $\Omega M$ further 
refines this, and allows for the calculation of a significant subgroup of 
the homotopy groups of the configuration space. 

To present our results, we start with a classification theorem. 
Assume that homology is taken with integral 
coefficients and use the symbol ``$\simeq$" to denote a homotopy equivalence. 
By a connected sum of sphere products, we mean a connected sum of 
products of two spheres. 

\begin{theorem} 
   \label{classify} 
   The following hold: 
   \begin{letterlist} 
      \item if $M$ and $N$ are simply-connected four-manifolds, then 
               $\Omega M\simeq\Omega N$ if and only if 
               $H^{2}(M)\cong H^{2}(N)$; 
      \item if $M$ and $N$ are $(n-1)$-connected $2n$-dimensional manifolds  
                and $n\notin\{2,4,8\}$, then $\Omega M\simeq\Omega N$ if and only if 
                $H^{n}(M)\cong H^{n}(N)$; 
      \item if $M$ and $N$ are $n$-dimensional connected sums of sphere 
                products, then $\Omega M\simeq\Omega N$ if and only if 
                 $H^{m}(M)\cong H^{m}(N)$ for each $m<n$. 
   \end{letterlist} 
\end{theorem} 

Observe that in each case, the homotopy type of $\Omega M$ depends 
only on the cohomology of $M$, regarded as a $\mathbb{Z}$-module, 
in degrees strictly less than the dimension of $M$. This contrasts with 
the situation before looping. For example, Milnor~\cite{Mi} proved that 
two simply-connected four-manifolds $M$ and $N$ are homotopy 
equivalent if and only if $M$ and $N$ have isomorphic cohomology 
rings. Theorem~\ref{classify} states that after looping the ring structure 
in cohomology plays no role, only the rank in degree~$2$ cohomology 
does. So looping considerably simplifies the homotopy types. This is 
interesting because $\Omega M$ has the same homotopy 
groups as $M$, just shifted down one dimension. We therefore 
immediately obtain the following corollary. 

\begin{corollary} 
   \label{htpygroups} 
   The following hold: 
   \begin{letterlist} 
      \item if $M$ and $N$ are simply-connected four-manifolds, then 
               $\pi_{\ast}(M)\cong\pi_{\ast}(N)$ if and only if 
               $H^{2}(M)\cong H^{2}(N)$; 
      \item if $M$ and $N$ are $(n-1)$-connected $2n$-dimensional manifolds  
                and $n\notin\{2,4,8\}$, then $\pi_{\ast}(M)\cong\pi_{\ast}(N)$ if and only if 
                $H^{n}(M)\cong H^{n}(N)$; 
      \item if $M$ and $N$ are $n$-dimensional connected sums of sphere 
                products, then $\pi_{\ast}(M)\cong\pi_{\ast}(N)$ if and only if 
                 $H^{m}(M)\cong H^{m}(N)$ for each $m<n$.
   \end{letterlist} 
   $\qqed$ 
\end{corollary}  

Part~(a) of Corollary~\ref{htpygroups} reproves a theorem of Duan and Liang~\cite{DL}
via more homotopy theoretic methods, while parts~(b) and~(c) generalize it to wider classes of manifolds.
Another proof of Part~(a), again using more geometric techniques, is given in a recent preprint by Basu and Basu~\cite{BB}.
In fact, they show that the result holds for stable homotopy groups in place of homotopy groups. 
It would be interesting to see whether this also holds for the generalizations presented here.

Theorem~\ref{classify} is proved by decomposing $\Omega M$ into 
a product of spaces, up to homotopy. Explicitly, we have the following. 

\begin{theorem} 
   \label{1conndecomp} 
   Let $M$ be a simply-connected four-manifold and suppose that 
   $\mbox{dim}\,H^{2}(M)=k$. If $k=0$ then $M\simeq S^{4}$, if $k=1$ 
   then $\Omega M\simeq S^{1}\times\Omega S^{5}$, and if $k\geq 2$ 
   then there is a homotopy equivalence 
   \[\Omega M\simeq S^1\times\Omega(S^{2}\times S^{3}) 
                  \times\Omega(J\vee (J\wedge \Omega(S^{2}\times S^{3})))\] 
   where $J=\bigvee_{i=1}^{k-1} (S^{2}\vee S^{3})$ if $k>2$ and $J=\ast$ if $k=2$. 
\end{theorem} 

\begin{theorem} 
   \label{Walldecomp} 
   Let $M$ be an $(n-1)$-connected $2n$-dimensional manifold and 
   suppose that \mbox{$\mbox{dim}\,H^{n}(M)=k$}. If $k\geq 2$ then there 
   is a homotopy equivalence 
   \[\Omega M\simeq\Omega(S^{n}\times S^{n})\times 
        \Omega (J\vee(J\wedge\Omega(S^{n}\times S^{n})))\] 
   where $J=\bigvee_{i=1}^{k-2} S^{n}$. 
\end{theorem}

\begin{theorem} 
\label{connsum} 
Let $M$ and $N$ be closed oriented $(m-1)$-connected $n$-dimensional manifolds, 
with $1<m\leq n-m$. Suppose that $H_*(M)$ is torsion-free and there is a ring isomorphism  
$H^*(N)\cong H^*(S^m\times S^{n-m})$. 
Let $M-\ast$ and $N-\ast$ be the punctured manifolds with a single point $\ast$ removed. 
Then the following hold: 
\begin{romanlist} 
\item there is a homotopy equivalence
$$
\Omega (M\#N)\simeq \Omega(S^{m}\times S^{n-m})\times 
    \Omega((M-\ast)\vee ((M-\ast)\wedge \Omega(S^{m}\times S^{n-m})));   
$$ 
\item the looped inclusion
$\seqm{\Omega((M-\ast)\vee\bar N)\simeq\Omega((M-\ast)\vee S^{m}\vee S^{n-m})} 
    {\Omega i}{\Omega M}$ 
has a right homotopy inverse. 
\end{romanlist} 
Consequently, the homotopy type of $\Omega (M\#N)$ is independent of 
the homotopy type of $N$, and depends only on the homotopy type of $M-\ast$. 
\end{theorem}

Theorems~\ref{Walldecomp} and~\ref{connsum} are consequences of much more general results presented in
Theorem~\ref{TMain} and Proposition~\ref{C1}, both of which are stated in the context of $CW$-complexes
and Poincar\'e duality spaces. 

Observe that in each of these theorems, the decompositions 
can be further refined. In each case, $J$ is a wedge of simply-connected 
spheres, so $J\simeq\Sigma J^{\prime}$ where $J^{\prime}$ is a wedge of 
spheres. Therefore, using the facts that 
$\Omega(X\times Y)\simeq\Omega X\times\Omega Y$, 
$\Sigma (X\times Y)\simeq\Sigma X\vee\Sigma Y\vee(\Sigma X\wedge Y)$ 
and by~\cite{J}, $\Sigma\Omega S^{s}\simeq\bigvee_{i=1}^{\infty} S^{(s-1)i+1}$, 
we see that $J\wedge\Omega(S^{s}\times S^{t})$ is homotopy equivalent to 
a wedge of spheres. Thus the factor 
$\Omega(J\vee(J\wedge\Omega(S^{s}\times S^{t})))$ 
is homotopy equivalent to the loops on a large wedge of spheres, 
and the Hilton-Milnor Theorem can be applied to decompose this as a 
product of loops on spheres of varying dimensions. In particular, 
in each case, $\Omega M$ decomposes as a product of loops on spheres, 
and so the homotopy groups of $M$ can be determined to the same 
extent as the homotopy groups of spheres. 
A similar refinement is possible in Theorem~\ref{connsum} when $M-\ast$
has the homotopy type of a suspension.

From this point of view, Theorems~\ref{1conndecomp}, \ref{Walldecomp} 
and~\ref{connsum} should be regarded as analogues of the 
Hilton-Milnor Theorem. As such, these theorems are very 
practical and should have numerous applications. We have already 
mentioned how they can be used to determine the homotopy groups 
of $M$. As another application described in detail in 
Section~\ref{sec:4manifold}, we consider principle $G$-bundles 
\(\seqm{P}{}{M}\), 
where $M$ is a simply-connected four-manifold and $G$ is a 
simply-connected, simple compact Lie group. It is well known 
that there are $[M,BG]\cong\mathbb{Z}$ distinct equivalence 
classes of such principle $G$-bundles. However, after looping 
the homotopy types of $\Omega P$ all coincide as 
$\Omega M\times\Omega G$. 

To prove Theorems~\ref{1conndecomp}, \ref{Walldecomp} 
and~\ref{connsum}, we consider a more general class 
of torsion-free $CW$-complexes which resemble Poincar\'{e} 
duality spaces. For such a space $P$ of connectivity $m-1$ and 
dimension~$n$, we assume that the $(n-1)$-skeleton $\bar P$ 
has $S^{m}\vee S^{n-m}$ as a wedge summand, and that there is a 
space $Q$ and a map 
$\seqm{P}{q}{Q}$ 
such that there is a ring isomorphism 
$H^{\ast}(Q)\cong H^{\ast}(S^{m}\times S^{n-m})$ and the composite 
\(\seqmm{S^{m}\vee S^{n-m}}{}{P}{}{Q}\)
is onto in cohomology. Taking $F$ to be the homotopy fibre of $q$,  
we analyze the homology of $F$ via the Serre spectral sequence, 
and then use this to determine its homotopy type. This is then fed 
into a decomposition of $\Omega P$ as $\Omega Q\times F$. The 
decompositions in the three theorems above then follow as 
special cases of this more general decomposition. All of this goes 
through provided there are no cup product squares in cohomology, 
which is the reason for the excluded cases $\{2,4,8\}$ in 
Theorem~\ref{Walldecomp}. In the case of simply-connected 
four-manifolds, these difficulties can be overcome through a 
novel modification. If the simply-connected four-manifold~$M$ 
is mapped into $\mathbb{C}P^{\infty}$ by representing a cohomology 
class in degree~$2$, then the homotopy fibre is a simply-connected 
$5$-dimensional Poincar\'{e} duality complex $Z$ which fits into the 
general class of torsion-free spaces $P$ above. The resulting 
decomposition of $\Omega Z$ is then used to determine the homotopy 
type of $\Omega M$.

\section{A general homotopy decomposition}

The loop space functor and localization functors both have effect of simplifying 
homotopy types while retaining most of the original homotopy theoretic information.
At one extreme a conjecture of Anick~\cite{Anick2} 
(for which there is some evidence~\cite{MW,Anick1,Stelzer1}) 
asserts that the loop space of any simply connected finite $CW$-complex
localized away from a predetermined finite set of primes decomposes 
as a weak product of a certain countable list of indecomposable spaces,  
while at the other end of the spectrum the loop space homotopy type of a 
highly connected $CW$-complex is uniquely determined. 
This is not difficult to see, for if $X$ and $Y$ and $(2n-2)$-dimensional $(n-1)$-connected,
then they are the $2n$-skeletons of $\Sigma\Omega X$ and $\Sigma\Omega Y$ respectively, 
so a homotopy equivalence $\Omega X\simeq\Omega Y$ would allow one to construct a composite 
\seqmmm{X}{incl.}{\Sigma\Omega X}{\simeq}{\Sigma\Omega Y}{eval.}{Y}
that induces an isomorphism on homology, and is therefore a homotopy 
equivalence. Recently, a much stronger result of Grbi\'{c} and Wu~\cite{GW} 
shows that if $X$ and $Y$ are simply-connected finite dimensional co-$H$ spaces  
then $X\simeq Y$ if and only if $\Omega X\simeq\Omega Y$. 

This leads to a natural question. Starting with a finite $CW$-complex $\bar P$,
and attaching a cell to $\bar P$ to form a space $P$, which homotopy classes of 
attaching maps yield the same loop space homotopy type for $P$?
By the above remarks, distinct homotopy classes of $co$-$H$-maps tend to yield 
distinct loop space homotopy types.  
Our goal is to provide sufficient cohomological criteria given a few conditions on $\bar P$.
More precisely, we give a loop space decomposition for a certain class 
of spaces, which includes certain connected-sums and certain 
Poincar\'{e} duality spaces (both examples to be discussed in more detail in the 
next section). Looping will have the effect of simplifying homotopy types, 
and the homotopy types of the loop spaces will be shown to depend only 
on simple data, often obtained from the homology of the original space 
in degrees strictly less than the dimension of the space. 
We begin by defining the class of spaces we have in mind. Throughout, 
homology is taken with integer coefficients. 

\begin{definition} 
\label{Pdef}  
Let $m$ and $n$ be integers such that $1<m\leq n-m$.
Suppose $P$ is a finite $n$-dimensional $(m-1)$-connected $CW$-complex
with torsion-free integral homology given by
$$
H_*(P)\cong \Z\paren{a_1,\ldots,a_\ell,z}  
$$
where 
$$
1<m=|a_1|\leq |a_2|\leq\cdots\leq |a_\ell|=n-m<|z|=n.
$$ 
Let $\bar P$ be the $(n-1)$-skeleton of $P$ and let 
$i\,\colon\seqm{\bar P}{}{P}$ 
be the skeletal inclusion. Notice that 
the bottom cell of $\bar P$ occurs in dimension $m$ while the 
top cell occurs in dimension $n-m$. 
\smallskip 

Define $\mathcal{P}$ as the collection of all such spaces $P$ which also 
satisfy the following two properties:  
\begin{itemize}
\item[($1$)] there is a homotopy equivalence $\bar P\simeq J\vee (S^{m}\vee S^{n-m})$ 
         for some space $J$; 
\item[($2$)] if $Q$ is the homotopy cofibre of the composite 
         $J\hookrightarrow\seqm{\bar{P}}{i}{P}$, then there is a ring isomorphism 
         $H^*(Q)\cong H^*(S^{m}\times S^{n-m})$. 
\end{itemize} 
\end{definition} 

To analyze $\Omega P$ for $P\in\mathcal{P}$, some observations and notation 
are required. 

\noindent 
\textbf{Observations}: 
\begin{itemize} 
\item[($1$)] If $X$ is a space and $H_*(X)$ is torsion-free, an element 
$x\in H_*(X)$ has a dual class in $H^*(X)$ which we label as $x^*$. 
In our case, since $H_*(P)$ is torsion-free, whenever $|a_i|+|a_j|=n$, 
the cup product $a_{i}^*a_{j}^*$ is some multiple of $z^*$; define 
the integer $c_{ij}$ by $a_{i}^*a_{j}^*=c_{ij}z^*$. 

\item[($2$)] Observe that the homological description of $P$ implies that there is a 
homotopy cofibration 
$$ 
\seqmm{S^{n-1}}{\alpha}{\bar P}{i}{P} 
$$ 
where $\alpha$ attaches the top cell to $P$. A basis for $H_*(\bar P)$ is given 
by the elements $\{a_1,\ldots,a_\ell\}$. 

\item[($3$)] The homotopy decomposition of $\bar P$ lets us define composites 
$$
s\,\colon J\hookrightarrow\seqm{\bar P}{i}{P}
$$
$$
s'\,\colon S^{m}\vee S^{n-m}\hookrightarrow\seqm{\bar P}{i}{P}.
$$ 
Let $\iota_{t}\in H_t(S^t)$ represent a generator. Without loss of 
generality we may assume that the basis for $H_*(P)$ has been chosen 
so that $(s')_*(\iota_{m})=a_{1}$ and $(s')_*(\iota_{n-m})=a_{\ell}$. 
Then the decomposition $\bar P\simeq J\vee(S^{m}\vee S^{n-m})$ implies 
that $s_*$ induces an injection onto $\{a_{2},\ldots,a_{\ell-1}\}$. 

\item[($4$)] The definition of $Q$ also lets us define a map $q$ by the homotopy 
cofibration 
$$ 
\seqmm{J}{s}{P}{q}{Q}. 
$$ 
As this cofibration induces a long exact sequence in homology, the fact 
that $\bar P\simeq J\vee (S^{m}\vee S^{n-m})$ is the $(n-1)$-skeleton of $P$ 
implies that the composite 
$\seqmm{S^m\vee S^{n-m}}{s'}{P}{q}{Q}$ 
induces an injection in homology. 

\item[($5$)] The ring isomorphism $H^*(Q)\cong H^*(S^{m}\times S^{n-m})$ 
implies that 
$$
H_*(Q)\cong \Z\{x,y,e\}, 
$$
where $|x|=m$, $|y|=n-m$, $|e|=n$ and the generators can be 
chosen so that $(x^*)^2=(y^*)^2=0$ and $y^*x^*=e^*$. Further, 
since $(q\circ s')_*$ is an injection, we have $q_*(a_{1})=x$,  
$q_{*}(a_{2})=y$ and $q_{*}(z)=e$; and as $q\circ s$ is null 
homotopic we have $q_{*}(a_{i})=0$ for $2\leq i\leq \ell-1$. 

\item[($6$)] The description of $q_{\ast}$ on the generators of $H_*(P)$ 
implies that $c_{\ell 1}=1$, $c_{1\ell}=(-1)^{m(n-m)}$, and 
$c_{11}=c_{\ell\ell}=0$. 
\end{itemize} 

We are aiming for the homotopy decomposition of $\Omega P$ 
stated in Theorem~\ref{TMain}. To get started, we begin with an initial 
decomposition. Define the space $F$ and the maps $f$ and $\delta$ 
by the homotopy fibration sequence
$$
\seqmmm{\Omega Q}{\delta}{F}{f}{P}{q}{Q}.
$$ 
We first calculate the homology of $\Omega Q$ and relate it to the 
homology of $\Omega (S^{m}\vee S^{n-m})$. By the Bott-Samelson 
theorem, there is an algebra isomorphism  
$$ 
H_*(\Omega(S^{m}\vee S^{n-m}))\cong T(u,v) 
$$ 
where $T(u,v)$ is the free tensor algebra on generators $u$ and $v$ 
of degrees $m-1$ and $n-m-1$ respectively. Let $\mathbb{Z}[u,v]$ 
be the polynomial algebra generated by $u$ and $v$. 

\begin{lemma} 
\label{Qhlgy} 
There is a coalgebra isomorphism 
$H_*(\Omega Q)\cong\mathbb{Z}[u,v]$ 
which can be chosen so that the map 
\(\llnamedright{\Omega(S^{m}\vee S^{n-m})}{\Omega (q\circ s')}{\Omega Q}\) 
induces in homology the abelianization 
$\seqm{T(u,v)}{}{\mathbb{Z}[u,v]}$. 
\end{lemma} 

\begin{proof} 
First, consider the homology Serre spectral sequence for the 
path-loop homotopy fibration 
$\seqmm{\Omega (S^{m}\vee S^{n-m})}{}{\ast}{}{S^{m}\vee S^{n-m}}$. 
Let $\iota_{k}\in H_{k}(S^{k})$ represent a generator. Then 
the elements $\iota_{m},\iota_{n-m}\in H_*(S^{m}\vee S^{n-m})$ 
transgress to the elements $u,v\in T(u,v)$, and the element 
$[u,v]\in T(u,v)$ arises in the spectral sequence as the element 
$u\otimes\iota_{n-m}+(-1)^{\vert u\vert\vert\iota_{n-m}\vert}\iota_{m}\otimes v$. 

Next, consider the homology Serre spectral sequence  
for the path-loop homotopy fibration 
$\seqmm{\Omega Q}{}{\ast}{}{Q}$. 
By Observation~(5), $H_*(Q)\cong\mathbb{Z}\{x,y,e\}$ where 
$\vert x\vert=m$, $\vert y\vert=n-m$, $\vert e\vert=n$ and 
the cohomology duals satisfy $x^*y^*=e^*$. Thus in homology, 
the reduced diagonal $\overline{\Delta}(e)$ equals $x\otimes y+y\otimes x$. 
Thus in the Serre spectral sequence for the path-loop homotopy fibration, we 
have $x$ and $y$ transgressing to elements $a$ and~$b$ respectively, and 
$d^{n}(e)=a\otimes y+(-1)^{\vert a\vert\vert y\vert} x\otimes b$. It is 
now a straightforward calculation to show that there is an isomorphism of 
vector spaces $H_*(\Omega Q)\cong\mathbb{Z}[a,b]$ where $\vert a\vert=m-1$ 
and $\vert b\vert=n-m-1$. 

Now consider the homotopy commutative diagram of path-loop homotopy fibrations 
\[\diagram 
         \Omega (S^{m}\vee S^{n-m})\rto\dto^{\Omega(q\circ s')} 
               & \ast\rto\dto & S^{m}\vee S^{n-m}\dto^{q\circ s'} \\ 
         \Omega Q\rto & \ast\rto & Q. 
  \enddiagram\] 
This induces a morphism of Serre spectral sequences between the 
two path-loop homotopy fibrations. By Observation~(5), the map $(q\circ s')_{\ast}$ 
is an isomorphism in degrees~$<n$. Therefore, comparing Serre spectral 
sequences, $(\Omega(q\circ s'))_{\ast}$ is an isomorphism in degrees~$<n-1$. 
In particular, $(\Omega(q\circ s'))_{\ast}$ is an isomorphism in degrees 
$m-1$ and $n-m-1$. Thus, up to sign, $(\Omega(q\circ s'))_{\ast}$ sends 
$u,v\in T(u,v)$ to $a,b\in\mathbb{Z}[u,v]$. Comparing spectral sequences, 
we also have the element 
$u\otimes\iota_{n-m}+(-1)^{\vert u\vert\vert\iota_{n-m}\vert}\iota_{m}\otimes v$ 
sent to $a\otimes y+(-1)^{\vert a\vert\vert y\vert} x\otimes b$, which is 
the image of the differential $d^{n}(e)$. That is, 
$[u,v]\in T(u,v)$ is sent to $0\in\mathbb{Z}[a,b]$. Further, it is 
straightforward to see that once the $d^{n}$ differential is taken 
into account and we move to $E^{n+1}$, that the $E^{n+1}$ page 
for the fibration 
$\seqmm{\Omega (S^{m}\vee S^{n-m})}{}{\ast}{}{S^{m}\vee S^{n-m}}$ 
maps onto the $E^{n+1}$ page for the fibration 
$\seqmm{\Omega Q}{}{\ast}{}{Q}$. As there are no more non-trivial 
differentials, the same is true of the $E^{\infty}$ pages, and so 
$(\Omega(q\circ s'))_{\ast}$ is onto. 

Finally, since $(\Omega(q\circ s'))_{\ast}$ is an algebra map and 
$(\Omega(q\circ s'))_{\ast}([u,v])=0$, there is a factorization 
\[\diagram 
         T(u,v)\rrto^-{(\Omega(q\circ s'))_{\ast}}\dto^{\pi} 
                & & H_*(\Omega Q)\cong\mathbb{Z}[a,b] \\ 
         \mathbb{Z}[u,v]\urrto_{g} & & 
  \enddiagram\] 
for some map $g$, where $\pi$ is the abelianization map. 
Since $(\Omega(q\circ s'))_{\ast}$ is onto and both 
$\mathbb{Z}[u,v]$ and $\mathbb{Z}[a,b]$ have the same 
Poincar\'{e} series, $g$ must be an isomorphism. The statement 
of the lemma now follows. 
\end{proof} 

By the Hilton-Milnor Theorem, the inclusion of the 
wedge into the product  
$\seqm{S^{m}\vee S^{n-m}}{j}{S^{n}\times S^{n-m}}$ 
has a right homotopy inverse after looping. That is, there is a map 
$$ 
\phi\,\colon\seqm{\Omega S^{n}\times\Omega S^{n-m}}{} 
     {\Omega(S^{m}\vee S^{n-m})} 
$$ 
which is a right homotopy inverse of $\Omega j$. 

\begin{lemma} 
\label{loopQsplit} 
The composite 
$\seqmmm{\Omega S^{m}\times\Omega S^{n-m}}{\phi}{\Omega(S^{m}\vee S^{n-m})} 
   {\Omega s'}{\Omega P}{\Omega q}{\Omega Q}$ 
is a homotopy equivalence. Consequently, in the homotopy fibration sequence 
\seqmmm{\Omega Q}{\delta}{F}{f}{P}{q}{Q},  
the map $\delta$ is null homotopic, 
implying that there are homotopy equivalences 
$$
\Omega P\simeq\Omega Q\times\Omega F\simeq 
    \Omega S^{m}\times\Omega S^{n-m}\times\Omega F. 
$$ 
\end{lemma} 

\begin{proof} 
The fact that $\phi$ is a right homotopy inverse of $\Omega i$ implies 
that $\phi_{\ast}$ is a coalgebra map which maps onto the sub-coalgebra 
$\mathbb{Z}[u,v]$ of $T(u,v)\cong H_*(\Omega(S^{m}\vee S^{n-m}))$. 
By Lemma~\ref{Qhlgy}, $(\Omega(q\circ s'))_{\ast}$ maps this sub-coalgebra 
isomorphically onto $H_{\ast}(Q)$. Thus $\Omega q\circ\Omega s'\circ\phi$ 
induces an isomorphism in homology and so is a homotopy equivalence. 

For the consequences, consider the homotopy fibration sequence 
$\seqmmm{\Omega F}{}{\Omega P}{\Omega q}{\Omega Q}{\delta}{F}$. 
We have just shown that $\phi\circ\Omega s'$ is a right homotopy inverse 
for~$\Omega q$. Therefore, the map $\delta$ 
is null homotopic, and this immediately implies that there is a 
homotopy equivalence $\Omega P\simeq\Omega Q\times\Omega F$.  
\end{proof} 

Next, we wish to give an explicit homotopy decomposition 
of the space $\Omega F$. The first step is to calculate its homology. 
By Observation~(4), the composite \seqmm{J}{s}{P}{q}{Q} is a homotopy 
cofibration, so it is null homotopic. Therefore, $s$ lifts through 
\seqm{F}{f}{P} 
to a map
$$
\bar s\colon\seqm{J}{}{F}. 
$$ 
By Observation~(3), $s_{\ast}$ induces an injection onto $\{a_{2},\ldots,a_{\ell-1}\}$.   
So its lift $\bar s$ has the property that~$(\bar s)_{\ast}$ is an 
injection, and we will also label a basis for the image of $(\bar s)_{\ast}$ by  
$\{a_{2},\ldots,a_{\ell-1}\}$. 

As the homotopy fibration 
\begin{equation}
\label{EPrincipalFib}
\seqmm{\Omega Q}{\delta}{F}{f}{P}. 
\end{equation} 
is principal, there exists a left action 
$$
\theta\,\colon\seqm{\Omega Q\times F}{}{F}
$$
such that the following diagram commutes up to homotopy
\begin{equation}
\label{EAction}
\diagram
\Omega Q\times \Omega Q\rto^{\ID\times\delta}\dto^{\mu} 
    & \Omega Q\times F\dto^{\theta}\\
\Omega Q\rto^{\delta} & F  
\enddiagram
\end{equation} 
where $\ID$ is the identity map and $\mu$ is the standard loop space 
multiplication. 

\begin{proposition}
\label{P4}
There is an isomorphism of left $H_{*}(\Omega Q)$-modules
$$
H_*(F)\cong \Z\paren{a_2,\ldots,a_{\ell-1}}\otimes H_*(\Omega Q),
$$
where $\mathbb{Z}\{a_{2},\ldots, a_{\ell-1}\}$ is the image of 
$\bar{s}_{\ast}$ and the left action of $H_{*}(\Omega Q)$ given by $\theta_*$. 
\end{proposition} 

\begin{proof} 
By a result of Moore~\cite{Mo}, the homology Serre spectral sequence $E$ 
for the principal homotopy fibration sequence \seqmm{\Omega Q}{\delta}{F}{f}{P} 
is a spectral sequence of left $H_*(\Omega Q)$-modules, with  
\begin{equation} 
\label{ESerre} 
E^{2}_{*,*}\cong H_*(P)\otimes H_*(\Omega Q).
\end{equation} 
Here, the left action is induced by $\theta_{\ast}$ and the 
differentials respect the left action of $H_*(\Omega Q)$. 
That is, up to sign, $d^n(f\otimes gh)=(1\otimes g)d^n(f\otimes h)$ whenever 
the differential $d^n$ is defined. We now proceed to calculate the spectral 
sequence. In doing so, it will be helpful to rewrite~(\ref{ESerre}) as 
\begin{equation} 
\label{ESerre2} 
E^{2}_{\ast,\ast}\cong\mathbb{Z}\{1,a_{1},\ldots,a_{\ell},z\}\otimes H_{\ast}(\Omega Q). 
\end{equation} 
\medskip 

\noindent 
\textit{Initial information on the differentials}. 
Consider the composite 
$\seqmm{S^{m}\vee S^{n-m}}{s'}{P}{q}{Q}$. 
By Observation~(4), $(q\circ s')_{\ast}$ is an injection in homology. 
The composite induces a homotopy fibration diagram 
\[\diagram 
     \Omega Q\rto\ddouble & Z\rto\dto 
          & S^{m}\vee S^{n-m}\rto^-{q\circ s'}\dto^{s'} & Q\ddouble \\ 
     \Omega Q\rto^-{\delta} & F\rto^-{f} & P\rto^-{q} & Q 
  \enddiagram\] 
which defines the space $Z$. Since $(q\circ s')_{\ast}$ is an injection 
in homology and there is a coalgebra isomorphism 
$H_{\ast}(Q)\cong H_{\ast}(S^{m}\times S^{n-m})$, 
in the homology Serre spectral sequence for the fibration 
$\seqmm{\Omega Q}{}{Z}{}{S^{m}\vee S^{n-m}}$ the generators 
$\iota_{m},\iota_{n-m}\in H_{\ast}(S^{m}\vee S^{n-m})$ transgress 
to the elements $u,v\in H_{\ast}(\Omega Q)$ respectively, where $u,v$ 
are as in Lemma~\ref{Qhlgy}. Now consider the homology Serre spectral 
sequence for the fibration 
$\seqmm{\Omega Q}{\delta}{F}{f}{P}$. 
By Observation~(3), we may assume that $(s')_{\ast}(\iota_{m})=a_{1}$ 
and $(s')_{\ast}(\iota_{n-m})=a_{l}$, 
so a comparison of spectral sequences implies that the elements 
$a_{1},a_{l}$ transgress to $u,v\in H_{\ast}(\Omega Q)$. That is, 
in terms of $E^{2}_{\ast,\ast}$, we have 
$$
d^{m}(a_1\otimes 1)=1\otimes u,\,\, d^{n-m}(a_\ell\otimes 1)=1\otimes v.
$$ 
Further, by Observation~(3), the map 
$\seqm{J}{s}{P}$ 
induces an injection in homology onto $\{a_{2},\ldots,a_{\ell-1}\}$, 
and it was observed before the statement of the proposition that 
the map $s$ lifts through $f$ to $F$. Therefore the elements $\{a_{2},\ldots,a_{\ell-1}\}$ 
survive the spectral sequence. Consequently, 
\begin{equation} 
  \label{dtinfo} 
  d^{t}(a_{i})=0\ \mbox{for all $t\geq 2$ and $2\leq i\leq\ell-1$}. 
\end{equation} 
\medskip 

\noindent 
\textit{Case 1: $m<n-m$}. For degree reasons, the differentials 
$d^{2},\ldots, d^{m-1}$ are all zero on the elements $a_{1},\ldots,a_{\ell}$, 
so the left action of $H_{\ast}(\Omega Q)$ implies that these differentials 
are identially zero. Therefore  
$$ 
E^{2}_{\ast,\ast}\cong E^{m}_{\ast,\ast}.  
$$ 

For $d^{m}$ we have $d^{m}(a_{1}\otimes 1)=1\otimes u$. The left 
action of $\theta_{\ast}$ implies that for any element $g\in H_{\ast}(\Omega Q)$, 
we have (up to sign), 
$$
d^{m}(a_1\otimes g)= (1\otimes g)d^{m}(a_1\otimes 1)= 
     (1\otimes g)(1\otimes u) = 1\otimes gu.
$$ 
By~(\ref{dtinfo}), $d^{m}(a_{i})=0$ for $2\leq i\leq\ell$. So 
the left action of $\theta_{\ast}$ implies that for 
$d^{m}(a_{i}\otimes g)=0$ for any $2\leq i\leq\ell$ and any $g\in H_*(\Omega Q)$. 
Next, consider the element $z\otimes 1$. 
Dualizing to the cohomology spectral sequence associated with $E$, 
we have for each $i$ such that $|a_i|=n-m$, 
\begin{align*}
d_{m}(a_i^*\otimes u^*)&=(d_{m}(a_i^*\otimes 1))(1\otimes u^*)+ 
    (-1)^{|a_i|}(a_i^*\otimes 1)d_{m}(1\otimes u^*)\\
&=(-1)^{\vert a_{i}\vert} (a_i^*\otimes 1)(a_1^*\otimes 1) = 
     (-1)^{\vert a_{i}\vert} c_{i1}(z^*\otimes 1).
\end{align*} 
This implies that in the homology Serre spectral sequence $E$ we have  
$$
d^{m}(z\otimes 1)=\csum{|a_i|=n-m}{}{(-1)^{\vert a_{i}\vert} c_{i1}(a_i\otimes u)}. 
$$
The left action of $\theta_{\ast}$ therefore implies that 
$$
d^{m}(z\otimes g)=\csum{|a_i|=n-m}{}{(-1)^{\vert a_{i}\vert} c_{i1}(a_i\otimes gu)}  
$$ 
for each $g\in H_*(\Omega Q)$. Therefore, as $c_{\ell 1}=1$ by Observation~(6),  
$c_{\ell 1}(a_\ell\otimes gu)=(a_\ell\otimes gu)$ is identified in $E^{m+1}_{n-m,*}$ 
with a linear combination of elements $a_i\otimes gu$ for $|a_i|=n-m$. Note 
that $a_{1}$ is excluded here since $\vert a_{1}\vert=m$ and in this case we 
have assumed that $m<n-m$. Collectively, we have determined the differential $d^{m}$, 
and obtain an isomorphism of left $H_{\ast}(\Omega Q)$-modules 
$$ 
E^{m+1}_{\ast,\ast}\cong\mathbb{Z}\{a_{2},\ldots,a_{\ell}\}\otimes H_{\ast}(\Omega Q). 
$$ 

Continuing, by~(\ref{dtinfo}), $d^{m+1},\ldots,d^{n-m-1}$ are all identically 
zero on the elements $a_{2},\ldots,a_{\ell-1}$ and for degree reasons, 
$d^{m+1},\ldots,d^{n-m-1}$ are all identically zero on $a_{\ell}$. So the left 
action of $\theta_{\ast}$ implies that these differentials are identically zero on 
all elements. Therefore there is an isomorphism 
$$ 
E^{m+1}_{\ast,\ast}\cong E^{n-m}_{\ast,\ast}. 
$$ 

For $d^{n-m}$, by~(\ref{dtinfo}), $d^{n-m}(a_{i})=0$ for 
$2\leq i\leq\ell -1$, so the left action of $\theta_{\ast}$ implies that 
$d^{n-m}(a_{i}\otimes g)=0$ for any $2\leq i\leq\ell-1$ and for 
any $g\in H_{\ast}(\Omega Q)$. From the initial calculation of differentials, 
we obtained $d^{n-m}(a_{\ell}\otimes 1)=1\otimes v$. The left 
action of $\theta_{\ast}$ therefore implies that for any element 
$g\in H_{\ast}(\Omega Q)$ we have (up to sign), 
$$
d^{n-m}(a_{\ell}\otimes g)= (1\otimes g)d^{n-m}(a_{\ell}\otimes 1)= 
     (1\otimes g)(1\otimes v) = 1\otimes gv.
$$ 
Thus we have determined the differential $d^{n-m}$, and obtain an isomorphism 
of left $H_{\ast}(\Omega Q)$-modules 
$$ 
E^{n-m+1}_{\ast,\ast}\cong\mathbb{Z}\{a_{2},\ldots, a_{\ell-1}\}\otimes H_{*}(\Omega Q). 
$$ 

Finally, by~(\ref{dtinfo}), the differentials $d^{t}$ for $t>n-m$ are all 
identically zero on $a_{2},\ldots,a_{\ell-1}$, so the left acton of $\theta_{\ast}$ 
implies that these differentials are identically zero on all elements. Hence 
$$ 
E^{\infty}_{\ast,\ast}\cong E^{n-m+1}_{\ast,\ast}.
$$ 
Since there is no torsion in $E^{\infty}_{*,*}$, there is no extension problem,
and we have
\begin{equation}
\label{EIso}
H_{\ast}(F)\cong \cplus{i+j=*}{}{E^{\infty}_{i,j}}
\cong\mathbb{Z}\{a_{2},\ldots,a_{\ell-1}\}\otimes H_{\ast}(\Omega Q).
\end{equation}
To see that this is an isomorphism of left $H_{*}(\Omega Q)$-modules,
recall that the left action 
\seqm{H_{\ast}(\Omega Q)\otimes E^{\infty}_{i,j}}{}{E^{\infty}_{i,j+\ast}}
coincides with the left action of associated graded objects
$$
\seqm{H_{\ast}(\Omega Q)\otimes \frac{\mc F_{i,i+j}}{\mc F_{i-1,i+j}}}{}
{\frac{\mc F_{i,i+j+\ast}}{\mc F_{i-1,i+j+\ast}}\cong E^{\infty}_{i,j+\ast}}
$$
induced by the action \seqm{H_{*}(\Omega Q)\otimes H_{i+j}(F)}{\mu_*}{H_{i+j+\ast}(F)}, 
where $\mc F_{i,j}=\mc F_i H_j(F)\subseteq H_j(F)$ is the increasing filtration
associated with our spectral sequence.
Observe from the calculations above that the action on the $E^{\infty}_{\ast,\ast}$ is free,
so the action on the associated graded objects is free.  
Therefore the action $\mu_*$ must also be free, 
and so the isomorphism~(\ref{EIso}) is one of left $H_{*}(\Omega Q)$-modules.

\medskip 

\noindent 
\textit{Case 2: $m=n-m$}. This case is simpler. 
We have $n=2m$ and $|u|=|v|=m-1$. So the only differential which 
comes into play is $d^{m}$. This time $d^{m}(z\otimes g)$ is the sum 
of linear combinations of the elements   
$c_{i1}(a_i\otimes gu)$ and $c_{i\ell}(a_i\otimes gv)$ for all $i$, 
where 
$c_{\ell 1}=1$, $c_{1\ell}=(-1)^{m(n-m)}=-(-1)^{|u||v|}$ and $c_{11}=c_{\ell\ell}=0$.
Therefore the elements $a_\ell\otimes gu-(-1)^{|u||v|}a_1\otimes gv$
are identified in $E^{m+1}_{m,*}$ with a linear combination of 
elements of the form $a_i\otimes gu$ or $a_i\otimes gv$ for $2\leq i\leq \ell-1$,
and the calculation goes through as before. 
\end{proof}
 
Now we refine the homotopy decomposition 
$\Omega P\simeq\Omega Q\times\Omega F$ 
of Lemma~\ref{loopQsplit} by identifying the homotopy type of $F$. 
For spaces $X$ and $Y$, the \emph{left half-smash} of $X$ and $Y$ is defined by 
$$ 
X\ltimes Y=(X\times Y)/(\ast\times Y).  
$$ 
It is well-known that if $Y$ is a suspension then there is a homotopy equivalence 
$$ 
X\ltimes Y\simeq Y\vee (X\wedge Y). 
$$ 

\begin{proposition} 
\label{Ftype} 
There is a homotopy equivalence 
$$ 
F\simeq \Omega Q\ltimes J. 
$$ 
\end{proposition} 

\begin{proof} 
Using the lift $\seqm{J}{\bar s}{F}$ of $\seqm{J}{s}{P}$ and the 
homotopy action $\seqm{\Omega Q\times F}{\theta}{F}$, define 
$\lambda$ as the composite
$$
\lambda\,\colon\seqmm{\Omega Q\times J}{\ID\times \bar s} 
    {\Omega Q\times F}{\theta}{F}.
$$ 
By~(\ref{EAction}), the restriction of $\theta$ to $\Omega Q$ is 
homotopic to $\delta$, which by Lemma~\ref{loopQsplit} is null homotopic. 
Therefore the composite 
$$
\seqmm{\Omega Q\times *}{\ID\times *}{\Omega Q\times J}{\lambda}{F}
$$
is null homotopic. Since the homotopy cofibre of $\ID\times\ast$ is 
$\Omega Q\ltimes F$, the map $\lambda$ extends to a map $\hat\lambda$ 
that makes the following diagram homotopy commute 
\[\diagram
\Omega Q\times *\rto^-{\ID\times *} 
& \Omega Q\times J\rto^{}\dto^{\lambda}
& \Omega Q\ltimes J\ar@{.>}[dl]^{\hat\lambda}\\
& F. 
\enddiagram\] 
By definition, $\lambda=\theta\circ(\ID\times\bar s)$, so Proposition~\ref{P4} 
implies that $\hat\lambda_{\ast}$ is an isomorphism. Thus $\hat\lambda$ 
is a homotopy equivalence. 
\end{proof} 

\begin{theorem}
\label{TMain}
Let $P\in\mathcal{P}$ and suppose that $P$ is $(m-1)$-connected 
and $n$-dimensional. Then the following hold: 
\begin{romanlist}
\item there is a homotopy equivalence  
$$
\Omega P\simeq \Omega(S^{m}\times S^{n-m})\times 
\Omega(\Omega(S^{m}\times S^{n-m})\ltimes J),  
$$
which, if $J$ is a suspension, refines to a homotopy equivalence  
$$
\Omega P\simeq \Omega(S^{m}\times S^{n-m})\times \Omega(J\vee (J\wedge \Omega(S^{m}\times S^{n-m}))); 
$$
\item the map $\seqm{\Omega\bar P}{\Omega i}{\Omega P}$ has a right 
homotopy inverse. 
\end{romanlist}
\end{theorem}

\begin{proof} 
For part~(i), by Lemma~\ref{loopQsplit}, $\Omega P\simeq\Omega Q\times\Omega F$ 
and $\Omega Q\simeq\Omega S^{m}\times\Omega S^{n-m}$, and by 
Proposition~\ref{Ftype}, $F\simeq \Omega Q\ltimes J$. Thus 
$$ 
\Omega P\simeq\Omega S^{m}\times\Omega S^{n-m}\times 
    \Omega((\Omega S^{m}\times\Omega S^{n-m})\ltimes J). 
$$ 
If $J$ is a suspension, this decomposition refines due to the fact  
that $\Omega Q\ltimes J\simeq J\vee (J\wedge\Omega Q)$. 

For part~(ii), define $\bar q$ as the composite 
$$
\bar q\,\colon\seqmm{\bar P}{i}{P}{q}{Q}. 
$$ 
From this composite we obtain a homotopy pullback diagram 
\begin{equation} 
\label{FbarFdgrm} 
\diagram 
\Omega Q\rto^{\bar\delta}\ddouble & \bar F\rto^{\bar f}\dto^{\tau} 
     & \bar P\rto^{\bar q}\dto^{i} & Q\ddouble\\
\Omega Q\rto^{\delta} & F\rto^{f} & P\rto^{q} & Q 
\enddiagram 
\end{equation} 
which defines the space $\bar F$ and the maps $\bar f$, $\bar\delta$ and $\tau$. 
In particular, this is a homotopy commutative diagram of principal fibration 
sequences, so if 
$\bar\theta\,\colon\seqm{\Omega Q\times\bar F}{}{\bar F}$ 
is the homotopy action for the top fibration sequence, then there is a 
homotopy commutative diagram of actions 
\[\diagram 
         \Omega Q\times\bar F\rto^-{\bar\theta}\dto^{\ID\times\tau} 
                & \bar F\dto^{\tau} \\   
        \Omega Q\times F\rto^-{\theta} & F. 
  \enddiagram\] 

By definition, the map $\seqm{J}{s}{P}$ factors as the composite 
$\seqmm{J}{r}{\bar P}{i}{P}$,  
where $r$ is the inclusion of the wedge summand in 
$\bar P\simeq J\vee (S^{m}\vee S^{n-m})$. Since $s$ lifts 
through $f$ to the map $\seqm{J}{\bar s}{F}$, the definition 
of $\bar F$ as a homotopy pullback in~(\ref{FbarFdgrm}) implies that there 
is a pullback map $\bar r\,\colon\seqm{J}{}{\bar F}$ such that 
$\bar f\circ\bar r\simeq r$ and $\tau\circ\bar r\simeq\bar s$. 
Combining this with the preceding diagram, we obtain a homotopy 
commutative diagram 
\begin{equation} 
\label{actiondgrm} 
  \diagram 
         \Omega Q\times J\rto^-{\ID\times\bar r}\ddouble 
                & \Omega Q\times\bar F\rto^-{\bar\theta}\dto^{\ID\times\tau} 
                & \bar F\dto^{\tau} \\   
        \Omega Q\times J\rto^-{\ID\times\bar s} & \Omega Q\times F\rto^-{\theta} & F. 
  \enddiagram 
\end{equation}  

By definition, the map $\seqm{S^{m}\vee S^{n-m}}{s'}{P}$ factors as 
the composite 
$\seqmm{S^{m}\vee S^{n-m}}{j}{\bar P}{i}{P}$, 
where $j$ is the inclusion of the wedge summand in 
$\bar P\simeq J\vee (S^{m}\vee S^{n-m})$. By Lemma~\ref{loopQsplit}, 
$\Omega(q\circ s')$ has a right homotopy inverse. As 
$\bar q=q\circ i$, we have $q\circ s'=q\circ i\circ j=\bar q\circ j$, 
so $\Omega(\bar q\circ j)$ has a right homotopy inverse. Consequently, 
$\Omega\bar q$ has a right homotopy inverse, which implies that in 
the homotopy fibration 
$\seqmm{\Omega P}{\Omega\bar q}{\Omega Q}{\bar\delta}{\bar F}$, 
the map $\bar\delta$ is null homotopic. 

Let $\bar\lambda$ be the composite along the top row of~(\ref{actiondgrm}), 
$$ 
\bar\lambda\,\colon\seqmm{\Omega Q\times J}{\ID\times\bar r} 
     {\Omega Q\times F}{\bar\theta}{\bar F}. 
$$ 
Since $\bar{\theta}$ is a homotopy action, its restriction to $\Omega Q$ 
is $\bar\delta$. Therefore the restriction of $\bar\lambda$ to $\Omega Q$ 
is $\bar\delta$, which is null homotopic. Thus there is a homotopy 
commutative diagram 
\[\diagram 
       \Omega Q\times\ast\rto^-{\ID\times\ast} 
          & \Omega Q\times J\rto\dto^{\bar\lambda} 
          & \Omega Q\ltimes J\ar@{.>}[dl]^{\tilde\lambda} \\ 
       & F & 
  \enddiagram\] 
where the top row is a homotopy cofibration and $\tilde\lambda$ 
is an extension of $\bar\lambda$. Now let $\gamma$ be the composite 
$$ 
\gamma\,\colon\seqmm{\Omega Q\ltimes\bar F}{\tilde\lambda}{\bar F}{\tau}{F}. 
$$ 
Observe that $\gamma$ is a choice of the extension $\hat\lambda$ in 
the proof of Proposition~\ref{Ftype}. Thus $\gamma$ induces an isomorphism 
in homology and so is a homotopy equivalence. Consequently, 
the map $\tau$ has a right homotopy inverse 
$\sigma\,\colon\seqm{F}{}{\bar F}$. 

Finally, consider the diagram 
\[\diagram 
        (\Omega S^{m}\times\Omega S^{n-m})\times\Omega F   
                    \rto^-{\ID\times\Omega\sigma}\drdouble   
             & (\Omega S^{m}\times\Omega S^{n-m})\times\Omega\bar F 
                    \rto^-{\Omega j\times\Omega\bar f}\dto^{\ID\times\Omega\tau} 
             & \Omega\bar P\times\Omega\bar P\rto^-{\mu}\dto^{\Omega i\times\Omega i} 
             & \Omega\bar P\dto^{\Omega i} \\ 
        & (\Omega S^{m}\times\Omega S^{n-m})\times\Omega F 
                   \rto^-{\Omega s'\times\Omega f} 
              & \Omega P\times\Omega P\rto^-{\mu} & \Omega P 
  \enddiagram\] 
where $\mu$ is the standard loop multiplication. The left triangle homotopy 
commutes since $\sigma$ is a right homotopy inverse of $\tau$. The 
middle square homotopy commutes since, by definition, $s'=i\circ j$, 
and by~(\ref{FbarFdgrm}), $f\simeq i\circ\bar f$. The right square homotopy 
commutes since $\Omega i$ is a loop map. By part~(i), the bottom row is a 
homotopy equivalence, so the homotopy commutativity of the diagram 
implies that~$\Omega i$ has a right homotopy inverse. 
\end{proof}

\section{Consequences}

In this section we apply Theorem~\ref{TMain} to two classes of examples, 
first to certain connected sums, and then to certain Poincar\'{e} duality 
complexes, and prove Theorems~\ref{Walldecomp}. 

If $M$ is a closed oriented $n$-dimensional manifold, let $\bar M$ 
be the $(n-1)$-skeleton of $M$. 
In particular, $M$ is homotopy equivalent to $M-\ast$,
and is obtained from $\bar M$ by attaching a single $n$-cell. 
Observe that if $N$ is a closed oriented 
$n$-dimensional manifold and there is a ring isomorphism 
$H^*(N)\cong H^*(S^{m}\times S^{n-m})$, then $\bar N\simeq S^{m}\vee S^{n-m}$. 
Denote the connected sum of two closed oriented $n$-dimensional 
manifolds $M$ and $N$ by $M\#N$. Observe that the $(n-1)$-skeleton 
of $M\#N$ is homotopy equivalent to $\bar M\vee\bar N$. Let 
$$ 
i\,\colon\seqm{\bar M\vee\bar N}{}{M\#N} 
$$ 
be the skeletal inclusion.

\begin{proof}[Proof of Theorem~\ref{connsum}]
We will show that $M\#N\in\mathcal{P}$. Let 
$P=M\#N$, let $\bar P$ be the $(n-1)$-skeleton of $P$, and let 
$\seqm{\bar P}{i}{P}$ be the skeletal inclusion. By the definitions 
of $M$ and $N$, $P$ is an $(m-1)$-connected, $n$-dimensional 
$CW$-complex. Since $P=M\#N$ is a closed oriented manifold, it satisfies 
Poincar\'{e} duality,  which implies that $\bar P\simeq\bar M\vee\bar N$ 
is actually $(n-m)$-dimensional. Note that as $m>1$ we have $n-m<n-1$, 
so $H_*(P)$ is torsion-free if and only if $H_*(\bar P)$ is torsion-free. 
But as $H_*(M)$ is torsion-free, so is $H_*(\bar M)$, which implies that 
$\bar P\simeq\bar M\vee\bar N\simeq\bar M\vee (S^{m}\vee S^{n-m})$ 
also has $H_*(\bar P)$ torsion-free. Thus $H_*(P)$ is torsion-free. 

Now if $J=\bar M$ then as $\bar N\simeq S^{m}\vee S^{n-m}$, we have 
$\bar P\simeq J\vee(S^{m}\vee S^{n-m})$. Let $Q$ be the cofibre of the 
composite $\seqmm{J}{}{\bar P}{i}{P}$, that is, $Q$ is the cofibre of the 
composite $\seqmm{\bar M}{}{\bar M\vee\bar N}{i}{M\#N}$. Then 
$Q\simeq N$, which implies that 
$H^*(Q)\cong H^*(N)\cong H^*(S^{m}\times S^{n-m})$. 
Thus $P=M\#N$ satisfies all the conditions of Definition~\ref{Pdef}, 
so $P\in\mathcal{P}$. The assertions of the proposition are now all 
direct applications of Theorem~\ref{TMain}.  
\end{proof}

\begin{example}
As an example of Theorem~\ref{connsum} in action, recall that 
an $n$-dimensional manifold $M$ is a connected sum of sphere products if 
$$
M\cong (S^{m_{1}}\times S^{n-m_{1}})\#\cdots\#(S^{m_{k}}\times S^{n-m_{k}})
$$ 
for some integers $m_{1},\ldots,m_{k}$. 
Let $M_{1}=(S^{m_{1}}\times S^{n-m_{1}})\#\cdots\#(S^{m_{k-1}}\times S^{n-m_{k-1}})$ 
and $N=S^{m_{k}}\times S^{n-m_{k}}$ so that $M=M_{1}\# N$. Observe that  
$\bar M_{1}=\bigvee_{i=1}^{k-1}(S^{m_{i}}\vee S^{n-m_{i}})$. 
So by Theorem~\ref{connsum}, there is a homotopy equivalence 
\[\Omega M\simeq\Omega(M_{1}\#N)\simeq\Omega(S^{m_{k}}\times S^{n-m_{k}})\times 
      \Omega(\bar M_{1}\vee(\bar M_{1}\wedge\Omega(S^{m_{k}}\times S^{n-m_{k}}))).\] 
 
\end{example}

Recall that $P$ is a \emph{Poincar\'e duality complex} if it has the homotopy type of 
a finite $CW$-complex and its cohomology ring $H^*(P;R)$ satisfies
Poincar\'e duality for all coefficient rings $R$. In particular every
oriented simply-connected manifold is a Poincar\'e duality complex.  

\begin{proposition}
\label{C1}
Fix $1<m\leq n$. If $m=n-m$, assume that $m\notin\{2,4,8\}$.
Let $P$ be an $(m-1)$-connected $n$-dimensional Poincar\'e duality complex   
such that $(n-1)$-skeleton $\bar P$ of $P$ has the homotopy type of a wedge 
of spheres. Let $i\,\colon\seqm{\bar P}{}{P}$ be the skeletal inclusion. Then 
the following hold: 
\begin{romanlist} 
\item there is a homotopy equivalence 
$$
\Omega P\simeq \Omega(S^{m}\times S^{n-m})\times \Omega(J\vee (J\wedge \Omega(S^{m}\times S^{n-m})))  
$$
where $J$ is obtained from $\bar P$ by quotienting out a copy of $S^{m}\vee S^{n-m}$; 
\item the map 
$\seqm{\Omega\bar P}{\Omega i}{\Omega P}$ 
has a right homotopy inverse. 
\end{romanlist} 
Consequently, the homotopy type of $\Omega P$ depends only on the 
homotopy type of~$\bar P$. 
\end{proposition}

We will need a preliminary lemma about the cohomology ring of Poincar\'e duality
complexes before we can prove this. 

\begin{lemma}
\label{LPD}
Let $P$ be an $n$-dimensional Poincar\'e duality complex such that $H_*(P)$ 
is torsion-free, and let $e^*$ be a generator of $H^n(P)\cong\mathbb{Z}$. 
Then for any positive integer $i\leq n$ and basis element $x^*$ in $H^i(P)$, 
there exists a choice of basis for $H^{n-i}(P)$ such that $x^*y^* = e^*$ for 
some $y^*$ in this basis.
\end{lemma}

\begin{proof} 
Let $x$ and $e$ be the homology duals of $x^*$ and $e^*$.
Since $H^*(P)$ satisfies Poincar\'e duality, the cap product homomorphism
$$
\seqm{e\cap H^i(P)}{}{H_{n-i}(P)}
$$
is an isomorphism, so it maps a basis of $H^i(P)$ to a basis of $H_{n-i}(P)$. 
Therefore
$$
y = e\cap x^*
$$
is an element in a basis for $H_{n-i}(P)$.

Since $H_*(P)$ is torsion-free, the cup product is dual to the cap product.
That is, there is a commutative diagram
\[\diagram
H^{n-i}(P)\rto^-{\cong}\dto^{\cup x^*} & Hom(H_{n-i}(P),Z)\dto^{(\cap x^*)^*}\\
H^n(P)\rto^-{\cong} & Hom(H_n(P),Z).
\enddiagram\]
In particular, since the homomorphism $(\cap x^*)$ sends $e$ to $y$ 
and $e$ generates $H^n(P)$, its dual $(\cap x^*)^* = (\cup x^*)$ sends $y^*$ to $e^*$, 
so we have
$$
y^*\cup x^* = e^*. 
$$ 
Since $y$ is an element in a basis for $H_{n-i}(P)$,
$y^*$ is an element in the dual basis for $H^{n-i}(P)$, and we are done.
\end{proof}

Note that if $m=n$ and $m\notin\{2,4,8\}$ then the element $y^*$ 
in Lemma~\ref{LPD} is \emph{not} equal to $\pm x^*$. But if $m\in\{2,4,8\}$ 
then we may have $y^*=\pm x^*$. This is the reason for the exclusion of 
this case in the statement of Proposition~\ref{C1}. 
\medskip 

\begin{proof}[Proof of Proposition~\ref{C1}] 
We will check that $P\in\mathcal{P}$. By Poincar\'e 
duality, $\bar P$ is $(n-m)$-dimensional. So as $m>1$, we have 
$n-m<n-1$, implying that $H_*(P)$ is torsion-free if and only if 
$H_*(\bar P)$ is torsion-free. But as $\bar P$ is homotopy equivalent 
to a wedge of spheres, $H_*(\bar P)$ is torsion-free and therefore 
$H_*(P)$ is torsion-free. 

Fix $e^*$ as a generator of $H^n(P)\cong\mathbb{Z}$. 
Let $x^*\in H^m(P)$ be a basis element. By Lemma~\ref{LPD} there 
exists a basis element $y^*\in H^{n-i}(P)$ such that $x^*y^*=e^*$. 
Since $\bar P$ is homotopy equivalent to a wedge of spheres,
$x^*$ and $y^*$ are spherical classes represented by maps 
\seqm{S^{m}}{\alpha}{\bar P} and \seqm{S^{n-m}}{\beta}{\bar P} 
and the wedge sum 
$\seqm{S^{m}\vee S^{n-m}}{\alpha+\beta}{\bar P}$ 
has a left homotopy inverse. Thus $\bar P\simeq J\vee (S^{m}\vee S^{n-m})$ 
where $J$ is the homotopy cofibre of $\alpha+\beta$. 
Let $Q$ be the homotopy cofibre of the composite 
$\seqmm{J}{}{\bar P}{i}{P}$. The homotopy equivalence for $\bar P$ 
and the fact that, as a $CW$-complex, $P=\bar P\cup e^{n}$ implies 
that $Q$ is a three-cell complex, $Q=(S^{m}\vee S^{n-m})\cup e^{n}$,  
and the map to the cofibre, $\seqm{P}{q}{Q}$, is onto in homology. 
Dualizing, $q^*$ is an injection. Suppose that $u*\in H^m(Q)$, 
$v^*\in H^{n-m}(Q)$ and $z^*\in H^{n}(Q)$ satisfy  $q^*(u^*)=x^*$, 
$q^*(v^*)=y^*$ and $q^*(z^*)=e^*$. Then the fact that $x^*y^*=e^*$ 
implies that $u^*v^*=z^*$. Thus there is a ring isomorphism 
$H^*(Q)\cong H^*(S^{m}\times S^{n-m})$. Thus $P$ satisfies all 
the conditions of Definition~\ref{Pdef}, so $P\in\mathcal{P}$. The assertions 
of the proposition are now all direct applications of Theorem~\ref{TMain}. 
\end{proof} 

As an example of Proposition~\ref{LPD} in action, we prove 
Theorem~\ref{Walldecomp}. Let $M$ be an $(n-1)$-connected 
$2n$-dimensional manifold. Observe that the $(2n-1)$-skeleton 
of $M$ is homotopy equivalent to $\bigvee_{i=1}^{k} S^{n}$, where 
$k=\mbox{dim}\,H^{n}(M)$. We aim to decompose $\Omega M$. 

\begin{proof}[Proof of Theorem~\ref{Walldecomp}] 
If $n\notin\{2,4,8\}$ and $k\geq 2$, then by 
Proposition~\ref{LPD} there is a homotopy equivalence 
\[\Omega M\simeq\Omega(S^{n}\times S^{n})\times 
     \Omega (J\vee(J\wedge\Omega(S^{n}\times S^{n})))\] 
where $J=\bigvee_{i=1}^{k-2} S^{n}$. 
\end{proof}

\section{The case of simply-connected $4$-manifolds} 
\label{sec:4manifold} 

Proposition~\ref{C1} does not cover the cases of simply-connected $4$-manifolds,
$3$-connected $8$-manifolds, or $7$-connected $16$-manifolds, 
due to the potential presence of nonzero cup product squares. To handle 
the case of simply-connected $4$-manifolds and prove 
Theorem~\ref{1conndecomp}, we use the fact that such 
spaces appear as the base space in a certain $S^1$ homotopy fibration 
whose total space is a Poincar\'e duality complex. These homotopy fibrations 
generalize the fiber bundle 
\seqmm{S^1}{}{S^5}{}{\mb CP^2}. 

Let $M$ be a simply-connected oriented $4$-manifold. If $H^2(M)=0$ 
then $M$ is homotopy equivalent to $S^{4}$, and the homotopy type 
of $\Omega S^{4}$ is well known to be $S^{3}\times\Omega S^{7}$. 
So we will assume from now on that $H^{2}(M)\neq 0$. Then, up to homotopy 
equivalence, there is a homotopy cofibration 
$$
\seqmm{S^3}{\alpha}{\bigvee_{i=1}^{k} S^2}{}{M} 
$$
for some map $\alpha$. Suppose that there is an isomorphism 
of $\mathbb{Z}$-modules 
$$
H^{*}(M)\cong\Z\paren{x_1,\ldots,x_k,z} 
$$ 
where $|x_i|=2$ and $|z|=4$.
Let $c_{ij}$ be such that $x_i x_j = c_{ij}z$.
Let $C$ be the $k\times k$ matrix
$$
C=\sbracket{c_{ij}}. 
$$ 
The anti-commutativity of the cup product implies that  
$c_{ij}=c_{ji}$, so $C$ is symmetric, and Poincar\'e duality 
implies that $C$ is nonsingular. 

Focus on the class $x_{k}\in H^{2}(M)$. By Lemma~\ref{LPD}, 
we may assume the basis of $H^{2}(M)$ has been chosen so that 
$c_{k\bar k}=1$ for some $\bar k$. That is, 
$$
x_k x_{\bar k} = z.
$$  
The cohomology class $x_{k}$ is represented by a map 
$$
q\colon \seqm{M}{}{K(\mathbb{Z},2)}. 
$$ 
Note that $K(\mathbb{Z},2)\simeq\mathbb{C}P^{\infty}$, and 
$\Omega\mathbb{C}P^{\infty}\simeq S^{1}$. Define the space $Z$ 
by the homotopy fibration sequence 
$$ 
\seqmmm{S^{1}}{}{Z}{}{M}{q}{\mathbb{C}P^{\infty}}. 
$$ 
A theorem of Quinn~\cite{Q} 
says that in a fibration of spaces having the homotopy type of 
finite $CW$-complexes, the total space is a Poincar\'e duality complex if 
and only if the fiber and base space are Poincar\'e duality complexes.
This, of course, also holds for homotopy fibrations. Therefore, as we 
have a homotopy fibration \seqmm{S^{1}}{}{Z}{}{M} 
and both $S^{1}$ and $M$ are Poincar\'e duality complexes, then so is $Z$.

\begin{lemma} 
\label{Zprops} 
The Poincar\'e duality complex $Z$ satisfies the following:
\begin{romanlist} 
\item there is a homotopy cofibration  
$$
\seqmm{S^4}{\gamma}{\bigvee_{i=1}^{k} (S^2\vee S^3)}{}{Z} 
$$ 
for some map $\gamma$; 
\item $H^*(Z)$ is torsion-free. 
\end{romanlist}
\end{lemma} 

\begin{proof} 
Consider the homotopy fibration 
$\seqmm{S^{1}}{}{Z}{}{M}$. 
We will use a Serre spectral sequence to calculate $H^*(Z)$. We have 
$E_{2}^{\ast,\ast}\cong H^*(S^1)\otimes H^*(M)$. 
Let $a\in H^{1}(S^{1})$ represent a generator and recall that, as a 
$\mathbb{Z}$-module, $H^*(M)\cong\mathbb{Z}\{x_{1},\ldots,x_{k},z\}$. 
Thus a $\mathbb{Z}$-module basis for $E_{2}^{\ast,\ast}$ is given by 
$$ 
\{1,1\otimes x_{1},\ldots, 1\otimes x_{k},1\otimes z, a\otimes 1, 
      a\otimes x_{1},\ldots,a\otimes x_{k},a\otimes z\}. 
$$ 
The fibration in question is induced by the map 
$\seqm{M}{q}{\mathbb{C}P^{\infty}}$ 
which represents the cohomology class~$x_{k}$. Therefore 
$d_{2}(a)=\pm x_{k}$. Changing the basis of $H^1(S^1)$ if need be, 
assume that $d_{2}(a)=x_{k}$. As $d_{2}$ is a differential, the fact 
that $x_{k}x_{\bar k}=z$ implies that 
$d_{2}(a\otimes x_{\bar k})=x_{k}x_{\bar{k}}=z$, while 
$d_{2}(a\otimes x_{j})=x_{k}x_{j}=c_{kj}z$. Thus a $\mathbb{Z}$-module 
basis for $E_{3}^{\ast,\ast}$ is given by 
$$ 
\{1,1\otimes x_{1},\ldots, 1\otimes x_{k-1}, (a\otimes x_{1}-a\otimes c_{k1}x_{\bar k}), 
   \ldots, (a\otimes x_{\bar k-1}-a\otimes c_{k(\bar k-1)}x_{\bar k}), 
$$ 
$$ 
\hspace{2cm} 
   (a\otimes x_{\bar k+1}-a\otimes c_{k(\bar k+1)}x_{\bar k}), 
   (a\otimes x_{k}-a\otimes c_{kk}x_{\bar k}), a\otimes z\}. 
$$ 
All other differentials are trivial for degree reasons, so we have 
$H^*(Z)\cong E_{\infty}^{\ast,\ast}\cong E_{3}^{\ast,\ast}$. 

Notice that the calculation for the rational cohomology Serre spectral 
is exactly the same. Thus the rationalization map 
$\seqm{H^*(Z;\mathbb{Z})}{}{H^*(Z;\mathbb{Q})}$ 
preserves the number of basis elements in each dimension. Thus 
$H^*(Z)$ is torsion-free, proving part~(ii). 

Notice that the description of $H^*(Z)$ implies that $Z$ has $k-1$ 
cells in dimension $2$ and $k-1$ cells in dimension $3$. The fact 
that $H^*(Z)$ is torsion-free therefore implies that the $3$-skeleton 
of $Z$ is homotopy equivalent to $\bigvee_{i=1}^{k} (S^{2}\vee S^{3})$. 
The one remaining nontrivial cell of $Z$ occurs in dimension~$5$, 
so $Z$ is the homotopy cofibre of a map 
$\seqm{S^{4}}{}{\bigvee_{i=1}^{k}(S^{2}\vee S^{3})}$, 
proving part~(i). 
\end{proof} 

\begin{remark} 
The space $Z$ is in fact a manifold, not just a Poincar\'{e} duality complex, which is diffeomorphic 
to the connected sum of $k$ copies of $S^{2}\times S^{3}$~\cite{DL}. 
As we only use the much simpler properties of $Z$ listed in Lemma~\ref{Zprops}, 
it is clarifying to leave the analysis of $Z$ as it stands in the statement and proof of the lemma. 
\end{remark} 

Before proceeding to decompose the loop space of a simply-connected $4$-manifold, 
we first decompose the loop space of the associated Poincar\'{e} duality space $Z$. Let 
$$ 
i\,\colon\seqm{\bigvee_{i=1}^{k-1}(S^{2}\vee S^{3})}{}{Z} 
$$ 
be the skeletal inclusion. 

\begin{proposition} 
\label{loopZ} 
If $k=1$ then $Z\simeq S^{5}$, so $\Omega Z\simeq\Omega S^{5}$. 
If $k\geq 2$ then the following hold: 
\begin{romanlist} 
\item there is a homotopy equivalence 
$$ 
\Omega Z\simeq\Omega(S^{2}\times S^{3})\times 
   \Omega(J\vee (J\wedge \Omega(S^{2}\times S^{3})))  
$$
where $J=\bigvee_{i=1}^{k-1} (S^{2}\vee S^{3})$ if $k>2$ and $J=\ast$ if $k=2$; 
\item the map $\seqm{\Omega(\bigvee_{i=1}^{k-1} (S^{2}\vee S^{3}))}{\Omega i} 
     {\Omega Z}$ 
has a right homotopy inverse. 
\end{romanlist} 
\end{proposition} 

\begin{proof} 
Notice that Proposition~\ref{Zprops}~(i) implies that if $k=1$ then 
$Z\simeq S^{5}$. Assume from now on that $k\geq 2$. We will show that 
the conditions of Proposition~\ref{C1} hold. The result of Quinn already 
cited implies that $Z$ is a Poincar\'{e} duality space, and by 
Proposition~\ref{Zprops}~(i), $Z$ is $1$-connected and 
$5$-dimensional. So with $m=2$ and $n=5$ we have $m=2<n-m=3$. 
By Proposition~\ref{Zprops}, the $4$-skeleton of $Z$ is homotopy equivalent 
to $\bigvee_{i=1}^{k-1}(S^{2}\vee S^{3})$. Thus $Z$ satisfies the hypotheses 
of Proposition~\ref{C1}, and applying the proposition immediately gives 
the statements of the proposition. 
\end{proof} 

We now prove Theorem~\ref{1conndecomp}, restated as follows. 

\begin{theorem}
\label{T1} 
Let $M$ be a simply-connected $4$-manifold and suppose 
$\mbox{dim}\, H^{2}(M)=k$ for $k>0$. If $k=1$ then there 
is a homotopy equivalence 
$$
\Omega M\simeq S^1\times\Omega S^5   
$$
and if $k\geq 2$ then there is a homotopy equivalence 
$$
\Omega M\simeq S^{1}\times\Omega Z\simeq S^1\times\Omega(S^{2}\times S^{3})\times 
   \Omega(J\vee (J\wedge \Omega(S^{2}\times S^{3})))  
$$
where $J=\bigvee_{i=1}^{k-1} (S^{2}\vee S^{3})$ if $k>2$ and $J=\ast$ if $k=2$. 
Consequently, the homotopy type of $\Omega M$ depends only on the 
integer $k=\dim\, H^{2}(M)$. 
\end{theorem}  

\begin{proof} 
Consider the map 
$\seqm{M}{q}{\mathbb{C}P^{\infty}}$ 
representing the cohomology class $x_{k}$. Since $M$ is simply-connected, 
any generator of $H_{2}(M)$ is in the image of the Hurewicz homomorphism. 
In our case, the homology class dual to $x_{k}$ is the Hurewicz image of a map 
$t\,\colon\seqm{S^{2}}{}{M}$. 
Dualizing, $t^{\ast}(x_{k})=\iota^*_{2}$, where $\iota^*_{2}$ 
is a generator of $H^2(S^2)$. Therefore, the composite 
$\seqmm{S^{2}}{t}{M}{q}{\mathbb{C}P^{\infty}}$ 
is degree one in cohomology. Let 
$\bar t\,\colon\seqm{S^{1}}{}{\Omega M}$ 
be the adjoint of $t$. Then the composite 
$\seqmm{S^{1}}{\bar t}{\Omega M}{\Omega q}{S^{1}}$ 
is degree one in cohomology, implying that it is a homotopy equivalence. 
Therefore, in the homotopy fibration 
$\seqmm{\Omega Z}{}{\Omega M}{\Omega q}{S^{1}}$, 
the map $\Omega q$ has a right homotopy inverse, implying that there 
is a homotopy equivalence 
$$ 
\Omega M\simeq S^{1}\times\Omega Z. 
$$ 
The theorem now follows from the decomposition of $\Omega Z$ 
in Proposition~\ref{loopZ}. 
\end{proof}

An analogue of Theorem~\ref{T1} holds for $3$-connected $8$-manifolds $M$,
provided that there is a map 
\seqm{M}{}{\mb HP^{2}} 
that induces a surjection onto $H^4(\mb HP^{2})\cong \Z$. 
In such a case, composing this map with the inclusion 
\seqm{\mb HP^{2}}{}{\mb HP^{\infty}} 
and then using the fact that $\mb HP^{\infty}\simeq S^3$, one obtains a 
principal homotopy fibration 
\seqmm{S^3}{}{Z}{}{M} with total space $Z$ an $11$-dimensional Poincar\'e duality complex. 
The only nonzero homology groups of $Z$ are in degrees $4$, $7$, and $11$, 
and using the associated action of $S^3$ on $Z$, it is not difficult to show that 
the $10$-skeleton of $Z$ is homotopy equivalent to a wedge of $4$-spheres 
and $11$-spheres. It is not really clear what may happen in the case of 
$7$-connected $16$-manifolds, as $S^{7}$ does not have a classifying space. 

We now re-organize the information appearing in the decomposition 
in Theorem~\ref{T1} when $k\geq 2$ to make it more clear how the 
decomposition depends on the $2$-skeleton of the $4$-manifold. Let 
$i\colon\seqm{\bigvee_{i=1}^{k} S^{2}}{}{M}$ 
be the skeletal inclusion. 
 
\begin{theorem} 
   \label{T2} 
   Let $M$ be a simply-connected $4$-manifold and suppose 
   $\mbox{dim}\, H^{2}(M)=k$ for $k\geq 2$. Then the map 
   $\seqm{\Omega(\bigvee_{i=1}^{k} S^{2})}{\Omega i}{\Omega M}$   
   has a right homotopy inverse. 
\end{theorem} 

\begin{proof} 
Recall that there is a homotopy fibration 
$\seqmm{Z}{r}{M}{q}{\mathbb{C}P^{\infty}}$. 
In Theorem~\ref{T1} it was shown that~$\Omega q$ has a right homotopy 
inverse, 
$f\colon\,\seqm{S^{1}}{}{\Omega M}$. Thus the composite 
\[\seqmm{S^{1}\times\Omega Z}{f\times\Omega r}{\Omega M\times\Omega M} 
     {\mu}{\Omega M}\] 
is a homotopy equivalence, where $\mu$ is the loop multiplication. 

By Proposition~\ref{loopZ}, the map 
$\seqm{\Omega(\bigvee_{s=1}^{k-1}(S^{2}\vee S^{3}))}{\Omega j}{\Omega Z}$ 
has a right homotopy inverse, where~$j$ is the inclusion of the $4$-skeleton 
into the $5$-dimensional space $Z$. Let 
$g\colon\,\seqm{\Omega Z}{}{\Omega(\bigvee_{s=1}^{k-1}(S^{2}\vee S^{3}))}$ 
be a right homotopy inverse of $\Omega j$. Let $h$ be the composite 
\[h\,\colon\seqmm{\bigvee_{s=1}^{k-1}(S^{2}\vee S^{3})}{j}{Z}{r}{M}.\] 
Then $\Omega h\circ g$ is homotopic to $\Omega r$. Therefore, 
by the previous paragraph, the composite 
\[\seqmmm{S^{1}\times\Omega Z}{f\times g} 
     {\Omega M\times\Omega(\bigvee_{s=1}^{k-1}(S^{2}\vee S^{3}))} 
     {\ID\times\Omega h}{\Omega M\times\Omega M}{\mu}{\Omega M}\] 
is a homotopy equivalence.  

Since $\bigvee_{s=1}^{k-1}(S^{2}\vee S^{3})$ is $3$-dimensional, 
the map $h$ factors through the $3$-skeleton of $M$, which is homotopy 
equivalent to $\bigvee_{i=1}^{k} S^{2}$. Thus $h$ factors as a composite 
$\seqmm{\bigvee_{s=1}^{k-1}(S^{2}\vee S^{3})}{h^{\prime}}{\bigvee_{i=1}^{k} S^{2}}{i}{M}$  
for some map $h^{\prime}$. Also, for connectivity and dimension reasons, the map 
$\seqm{S^{1}}{f}{M}$ 
factors as a composite 
$\seqmm{S^{1}}{f^{\prime}}{\Omega(\bigvee_{i=1}^{k} S^{2})}{\Omega i}{\Omega M}$ 
for some map $f^{\prime}$. Therefore, inserting these factorizations into the 
homotopy equivalence $\mu\circ(\ID\times\Omega h)\circ(f\times g)$, we 
obtain a homotopy equivalence 
\[\seqmmm{S^{1}\times\Omega Z}{f^{\prime}\times g} 
     {\Omega(\bigvee_{s=1}^{k-1} S^{2})\times\Omega(\bigvee_{s=1}^{k-1}(S^{2}\vee S^{3}))} 
     {\ID\times\Omega h^{\prime}} 
     {\Omega(\bigvee_{i=1}^{k} S^{2})\times\Omega(\bigvee_{i=1}^{k} S^{2})} 
     {\Omega i\times\Omega i}{\Omega M\times\Omega M}\stackrel{\mu}{\longrightarrow} 
     \Omega M.\] 
Finally, since $\Omega i$ is a loop map, it commutes with the loop multiplication, 
so we obtain a homotopy equivalence 
\[\llnamedddright{S^{1}\times\Omega Z}{f^{\prime}\times(\Omega h^{\prime}\circ g)} 
        {\Omega(\bigvee_{i=1}^{k} S^{2})\times\Omega(\bigvee_{i=1}^{k} S^{2})}{\mu} 
        {\Omega(\bigvee_{i=1}^{k} S^{2})}{\Omega i}{\Omega M}.\] 
Consequently, the map $\Omega i$ has a right homotopy inverse. 
\end{proof} 

Theorem~\ref{T2} is useful. For example, we apply it to determine 
the homotopy type of the loops on certain principle $G$-bundles. 

\begin{corollary} 
   \label{bundlecor} 
   Let $G$ be a simply-connected, simple compact Lie group. Let $M$ 
   be a simply-connected $4$-manifold with $\mbox{dim}\, H^{2}(M)\geq 2$. 
   Let $\seqm{P}{\pi}{M}$ be a principle $G$-bundle. Then $\Omega\pi$ has 
   a right homotopy inverse, implying that there is a homotopy 
   equivalence 
   \[\Omega P\simeq\Omega M\times\Omega G.\] 
\end{corollary} 

\begin{proof} 
Any principle $G$-bundle 
$\seqm{P}{\pi}{M}$ 
is classified by a map 
$\seqm{M}{g}{BG}$, 
where $BG$ is the classifying space of $G$ and $P$ is the homotopy 
fibre of $g$. In our case, since $G$ is a simply-connected, 
compact simple Lie group, $BG$ is $2$-connected (in fact, it is $3$-connected). 
Thus the composite 
$\seqmm{\bigvee_{i=1}^{k} S^{2}}{i}{M}{g}{BG}$ 
is null homotopic by connectivity. By Theorem~\ref{T2}, $\Omega i$ 
has a right homotopy inverse. Therefore $\Omega g$ is null homotopic. 
Hence in the homotopy fibration sequence 
$\seqmmm{\Omega G}{}{\Omega P}{\Omega\pi}{\Omega M}{\Omega g}{G}$  
the null homotopy for $\Omega g$ implies that $\Omega\pi$ has a right 
homotopy inverse, and therefore $\Omega P\simeq\Omega M\times\Omega G$. 
\end{proof} 

Corollary~\ref{bundlecor} says something interesting. While there 
are $[M,BG]\cong\mathbb{Z}$ distinct principle $G$-bundles over $M$, 
after looping all those bundles become homotopy equivalent. 
Further, the decomposition of $\Omega P$ can be refined by 
inserting the decomposition of $\Omega M$ in Theorem~\ref{T1}, 
and - after localizing at a prime $p$ - by the decompositions of 
$\Omega G$ that arise from the $p$-local decompostions 
of~$G$ due to Mimura, Nishida and Toda~\cite{MNT}.

\section{Looped Configuration Spaces}
\label{sec:config} 

We end with a quick application that is in the spirit of our previous results. Let 
$$
F_k(X) = \{(x_1,\ldots,x_k)\in X^{\times k}\mathrel{|}x_i\neq x_j\mbox{ if }i\neq j\}
$$
be the \emph{ordered configuration space} of $k$ distinct points in $X$. 
The literature on these spaces is substantial, but many basic questions remain unanswered. 
For example, their integral homology is not clearly understood in most cases, 
and it is now known that their homotopy type generally does not depend only 
on the homotopy type of $X$,
even after restricting the input space to compact manifolds~\cite{LS}. 

Things do simplify after looping however. 
If we were to take $M$ to be a smooth manifold with a nonvanishing tangent vector field,
then the projection map \seqm{F_k(M)}{}{M} onto the first coordinate has a section.
By~\cite{FN,CG} there is a homotopy decomposition
\begin{equation}
\label{EConfigDecomp}
\Omega F_k(M)\simeq \Omega M\times\Omega(M-Q_1)\times\cdots\times\Omega (M-Q_k)               
\end{equation}
for any choice of distinct points $q_1,\ldots,q_k$ in $M$, with $Q_{i}=\{q_1,\ldots,q_i\}$. 
Thus, not only are the betti numbers $\Omega F_k(M)$ relatively easy to compute, 
but the homotopy type of $\Omega F_k(M)$ depends only on the homotopy type of the input manifold $M$ when $M$ is simply connected.
The following takes this a step further:

\begin{corollary}
Let $1<m\leq n-m$, $n$ be odd, and let $M$ be a closed oriented $(m-1)$-connected $n$-dimensional smooth manifold with torsion-free homology.
Then the homotopy type of the looped configuration space $\Omega F_k(M\#(S^m\times S^{n-m}))$ depends only on the homotopy type of $M-\ast$ for each $k\geq 1$.
\end{corollary}

\begin{proof}
Recall that the connected sum of smooth manifolds can be constructed so that the resulting manifold also has a smooth structure.
Then $M\#(S^m\times S^{n-m})$ is a smooth manifold, and moreover it is odd dimensional, so it has a nonvanishing tangent vector field.
Thus, the decomposition~(\ref{EConfigDecomp}) specializes to 
$$           
\Omega F_k(M\#(S^m\times S^{n-m}))\simeq\Omega(M\#(S^m\times S^{n-m}))\times\Omega(M\#(S^m\times S^{n-m})-Q_1)\times\cdots\times\Omega(M\#(S^m\times S^{n-m})-Q_k)
$$
for any choice of $k$ distinct points $q_1,\ldots,q_k$ in $M\#(S^m\times S^{n-m})$. 
 
Notice that $M\#(S^m\times S^{n-m})-Q_i$ is homotopy equivalent to the wedge sum of $(M-\ast)\vee S^m\vee S^{n-m}$ with $i-1$ copies of the $(n-1)$-sphere. 
Thus, the homotopy type of each factor $\Omega(M\#(S^m\times S^{n-m})-Q_i)$ in the  decomposition above depends only on the homotopy type of $M-\ast$.
Likewise, the homotopy type of the remaining factor $\Omega(M\#(S^m\times S^{n-m}))$ depends only on that of $M-\ast$ by Theorem~(\ref{connsum}).
The result follows. 
\end{proof}

\bibliographystyle{amsalpha}

\end{document}